\documentclass{amsart}
\usepackage{latexsym,amssymb,amsmath,amsthm,amscd,graphicx}
\usepackage{esint} 
\usepackage{color}
\usepackage{braket}
\numberwithin{equation}{section}
\newtheorem{theorem}{Theorem}[section]
\newtheorem{lemma}[theorem]{Lemma}

\newtheorem{proposition}[theorem]{Proposition}

\theoremstyle{definition}

\newtheorem{remark}[theorem]{Remark}
\newtheorem{definition}[theorem]{Definition}
\newtheorem{example}[theorem]{Example}

\theoremstyle{remark}

\usepackage{comment}

\usepackage{mathrsfs}

\title{Predual of weak Orlicz spaces}

\author{Naoya Hatano, Ryota Kawasumi, and Takahiro Ono}
\address[Naoya Hatano]{Department of Mathematics, Chuo University, 1-13-27, Kasuga, Bunkyo-ku, Tokyo 112-8551, Japan,}
\address[Ryota Kawasumi]{Minohara 1-6-3 (B-2), Misawa, Aomori, 033-0033}
\address[Takahiro Ono]{Department of Mathematics, Chuo University, 1-13-27, Kasuga, Bunkyo-ku, Tokyo 112-8551, Japan}

\email[Naoya Hatano]{n.hatano.chuo@gmail.com}
\email[Ryota Kawasumi]{rykawasumi@gmail.com}
\email[Takahiro Ono]{t.ono.tmu@gmail.com}

\begin{document}

\maketitle

\begin{abstract}
In this paper, we consider the  predual spaces of weak Orlicz spaces.
As an application, we provide the Fefferman-Stein vector-valued maximal inequality for the weak Orlicz spaces.
In order to prove this statement, we introduced the Orlicz-Lorentz spaces, and showed the boundedness of the Hardy-Littlewood maximal operator on these spaces.
\end{abstract}



\section{Introduction}\label{s1}

The purpose of this paper is to give the predual spaces of weak Orlicz spaces.
Moreover, as an application, we provide the Fefferman-Stein vector-valued maximal inequality for the weak Orlicz spaces.

For a function $\Phi:[0,\infty]\to[0,\infty]$, let
\begin{equation*}
a(\Phi)\equiv\sup\{t\ge0\,:\,\Phi(t)=0\}, \quad
b(\Phi)\equiv\inf\{t\ge0\,:\,\Phi(t)=\infty\}.
\end{equation*}

\begin{definition}[Young function]
An increasing function $\Phi:[0,\infty]\to[0,\infty]$ is called a Young function (or sometimes also called an Orlicz function) if it satisfies the following properties;
\begin{itemize}
\item[(1)] $0\le a(\Phi)<\infty$, $0<b(\Phi)\le\infty$,
\item[(2)] $\lim\limits_{t\to0}\Phi(t)=\Phi(0)=0$,
\item[(3)] $\Phi$ is convex on $[0,b(\Phi))$,
\item[(4)] if $b(\Phi)=\infty$, then $\lim\limits_{t\to\infty}\Phi(t)=\Phi(\infty)=\infty$,
\item[(5)] if $b(\Phi)<\infty$, then $\lim\limits_{t\to b(\Phi)-0}\Phi(t)=\Phi(b(\Phi))$.
\end{itemize}
\end{definition}

For $t>0$ and $f\in L^0({\mathbb R}^n)$, the distribution function $m(f,t)$ and the rearrangement function $f^\ast(t)$ are defined by
\begin{equation*}
m(f,t)\equiv|\{x\in{\mathbb R}^n:|f(x)|>t\}|,
\quad
f^\ast(t)\equiv\inf\{\alpha>0:m(f,\alpha)\le t\}.
\end{equation*}
Here it will be understood that $\inf\emptyset=\infty$.

\begin{definition}[weak Orlicz space]
For a Young function $\Phi:[0,\infty]\to[0,\infty]$, let
\begin{equation*}
{\rm w}L^\Phi({\mathbb R}^n)
\equiv
\left\{
f\in L^0({\mathbb R}^n)
:
\sup_{t>0}\Phi(t)m(kf,t)<\infty\,\text{for some $k>0$}
\right\},
\end{equation*}
\begin{equation*}
\|f\|_{{\rm w}L^\Phi}
\equiv
\inf
\left\{
\lambda>0
:
\sup_{t>0}\Phi(t)m\left(\frac f\lambda,t\right)\le1
\right\}.
\end{equation*}
\end{definition}

\begin{remark}
Let $f\in{\rm w}L^\Phi({\mathbb R}^n)$.
Then there exists $k>0$ such that
\begin{equation*}
\sup_{t>0}\Phi(t)m(kf,t)
\le1.
\end{equation*}
In fact, we put
\begin{equation*}
M\equiv
\sup_{t>0}\Phi(t)m(kf,t)
<\infty,
\end{equation*}
and assume that $M>1$.
Note that, by the convexity of $\Phi$,
\begin{equation*}
\frac1M\Phi(t)\ge\Phi\left(\frac tM\right),
\quad t\in[0,\infty].
\end{equation*}
Then, taking $0<k_0\le\dfrac1M$, we have
\begin{align*}
\sup_{t>0}\Phi(t)m(kk_0f,t)
&=
\sup_{t>0}\Phi(t)m\left(kf,\frac tk_0\right)
=
\sup_{t>0}\Phi(k_0t)m(kf,t)
\le
\sup_{t>0}\Phi\left(\frac tM\right)m(kf,t)\\
&\le
\sup_{t>0}\frac1M\Phi(t)m(kf,t)
=1.
\end{align*}
\end{remark}

A Young function $\Phi:[0,\infty]\to[0,\infty]$ is said to satisfy the $\Delta_2$-condition, denoted $\Phi\in\Delta_2$, if
\begin{equation*}
\Phi(2r)\le k\Phi(r)
\quad \text{for} \quad
r>0,
\end{equation*}
for some $k>1$.
A Young function $\Phi:[0,\infty]\to[0,\infty]$ is said to satisfy the $\nabla_2$-condition, denoted $\Phi\in\nabla_2$, if
\begin{equation*}
\Phi(r)\le\frac1{2k}\Phi(kr)
\quad \text{for} \quad
r\ge0,
\end{equation*}
for some $k>1$.

\begin{example}
\begin{itemize}
\item[(1)] $\Phi(t)=t^p$, $1\le p<\infty$, belongs to $\Delta_2$.
\item[(2)] $\Phi(t)=t^p$, $1<p<\infty$, belongs to $\nabla_2$.
\item[(3)] $\Phi(t)=t$ does not belong to $\nabla_2$.
\item[(4)] $\Phi(t)=t\log(3+t)$ belongs to $\Delta_2$, but does not belong to $\nabla_2$.
\item[(5)] $\Phi(t)=e^t-1$ belongs to $\nabla_2$, but does not belong to $\Delta_2$.
\end{itemize}
\end{example}

Let
\begin{equation*}\label{inverse function}
\Phi^{-1}(u)
\equiv
\begin{cases}
\inf\{t\ge0\,:\,\Phi(t)>u\}, & u\in[0,\infty),\\
\infty, & u=\infty.
\end{cases}
\end{equation*}
Then $\Phi^{-1}(u)$ is finite for all $u\in[0,\infty)$, continuous on $(0,\infty)$ and right continuous at $u=0$.
If $\Phi$ is bijective from $[0,\infty]$ to itself, then $\Phi^{-1}$ is the usual inverse
function of $\Phi$.
It is also known that
\begin{equation}\label{inverse equation}
\Phi(\Phi^{-1}(t))\le t\le\Phi^{-1}(\Phi(t)),
\quad t\in[0,\infty].
\end{equation}

For a Young function $\Phi:[0,\infty]\to[0,\infty]$, the complementary function is defined by
\begin{equation*}
\widetilde{\Phi}(r)
\equiv
\begin{cases}
\sup\{rs-\Phi(s):s\in[0,\infty)\}, & r\in[0,\infty),\\
\infty, & r=\infty.
\end{cases}
\end{equation*}
Then $\widetilde{\Phi}$ is also a Young function and $\widetilde{\widetilde{\Phi}}=\Phi$.
Note that $\Phi\in\nabla_2$ if and only if $\widetilde{\Phi}\in\Delta_2$.
It is known that
\begin{equation}\label{eq:201006-2}
r\le\Phi^{-1}(r)\widetilde{\Phi}^{-1}(r)\le2r
\quad \text{for} \quad
r\ge0.
\end{equation}

\begin{theorem}\label{thm:201006-1}
Let $\Phi:[0,\infty]\to[0,\infty]$ be a Young function such that $\Phi\in\Delta_2$.
Then
\begin{equation*}
L^{\Phi,1}({\mathbb R}^n)^\ast={\rm w}L^{\widetilde{\Phi}}({\mathbb R}^n)
\end{equation*}
with equivalence of quasi-norms, where the space $L^{\Phi,1}({\mathbb R}^n)$ is defined by the set of all measurable functions with the finite quasi-norm
\begin{equation*}
\|f\|_{L^{\Phi,1}}
\equiv
\int_0^\infty\Phi^{-1}\left(\frac1t\right)^{-1}f^\ast(t)\,\frac{{\rm d}t}t.
\end{equation*}
\end{theorem}

\begin{definition}[Hardy-Littlewood maximal operator]
For a measurable function $f$ defined on ${\mathbb R}^n$, define a function $Mf$ by
\begin{equation*}
Mf(x)
\equiv
\sup_{Q\in{\mathcal Q}({\mathbb R}^n)}\frac{\chi_Q(x)}{|Q|}\int_Q|f(y)|\,{\rm d}y,
\quad x\in{\mathbb R}^n,
\end{equation*}
where ${\mathcal Q}({\mathbb R}^n)$ denotes the family of all cubes with parallel to coordinate axis in ${\mathbb R}^n$.
\end{definition}

\begin{theorem}\label{thm:201007-1}
Let $\Phi:[0,\infty]\to[0,\infty]$ be a Young function, and let $1<q<\infty$.
\begin{itemize}
\item[(1)] If $\Phi\in\nabla_2$, then we have
\begin{equation*}
\left\|\sup_{j\in{\mathbb N}}Mf_j\right\|_{{\rm w}L^\Phi}
\lesssim
\left\|\sup_{j\in{\mathbb N}}|f_j|\right\|_{{\rm w}L^\Phi}
\end{equation*}
for any sequence $\{f_j\}_{j=1}^\infty$ of measurable functions.
\item[(2)] If $\Phi\in\Delta_2\cap\nabla_2$, then we have
\begin{equation*}
\left\|\left(\sum_{j=1}^\infty(Mf_j)^q\right)^{\frac1q}\right\|_{{\rm w}L^\Phi}
\lesssim
\left\|\left(\sum_{j=1}^\infty|f_j|^q\right)^{\frac1q}\right\|_{{\rm w}L^\Phi}
\end{equation*}
for any sequence $\{f_j\}_{j=1}^\infty$ of measurable functions.
\end{itemize}
\end{theorem}

We organize the remaining part of the paper as follows:
We prepare the statements for the proof of Theorem \ref{thm:201006-1} in Section \ref{s2}, and show Theorem \ref{thm:201006-1} in Section \ref{s3}.
Next, we provide the boundedness of the Hardy-Littlewood maximal operator on generalized Lorentz spaces in Section \ref{s4}.
Finally, we prove Theorem \ref{thm:201007-1} in Section \ref{s5}.

\section{Preliminaries}\label{s2}

\subsection{Statements of inverse Young function for $\Delta_2$ and $\nabla_2$ conditions}

\begin{lemma}\label{lem:201107-1}
Let $\Phi:[0,\infty]\to[0,\infty]$ be a Young function.
Then the followings are obtained:
\begin{itemize}
\item[(1)] $\Phi\in\Delta_2$ if and only if there exists $k>1$ such that
\begin{equation*}
\Phi^{-1}(ku)\ge2\Phi^{-1}(u).
\end{equation*}
\item[(2)] $\Phi\in\nabla_2$ if and only if there exists $k>1$ such that
\begin{equation*}
\Phi^{-1}(2ku)\le k\Phi^{-1}(u).
\end{equation*}
\end{itemize}
\end{lemma}

\begin{proof}
First we show (1).
We assume that $\Phi\in\Delta_2$.
Then
\begin{equation*}
\Phi(2\Phi^{-1}(u)) \le k \Phi(\Phi^{-1}u) \le ku
\end{equation*}
and therefore we have
\begin{equation*}
2\Phi^{-1}(u) \le \Phi^{-1}(ku).
\end{equation*}
Conversely, we assume that there exists $k>1$ such that
\begin{equation*}
\Phi^{-1}(ku)\ge2\Phi^{-1}(u).
\end{equation*}
For $u\ge0$, taking $t = \Phi(u)$, we have $u \le \Phi^{-1}(\Phi(u)) = \Phi^{-1}(t)$.
Hence
\begin{equation*}
\Phi(2u) \le \Phi(2 \Phi^{-1}(t)) \le \Phi( \Phi^{-1}(kt)) \le  kt \le k\Phi(u).
\end{equation*}

Next we show (2).
We assume that $\Phi\in\nabla_2$.
Then, for any $v\ge0$,
\begin{equation*}
\Phi\left(\frac1k\Phi^{-1}(v)\right) \le \frac1{2k}\Phi(\Phi^{-1}(v)) \le \frac{v}{2k},
\end{equation*}
and hence
\begin{equation*}
\frac1k\Phi^{-1}(v) \le \Phi^{-1}\left( \frac{v}{2k}\right).
\end{equation*}
Putting $v \mapsto 2ku$, weh have
\begin{equation*}
\Phi^{-1}(2ku) \le k\Phi^{-1}(u).
\end{equation*}
Conversely, we assume that there exists $k>1$ such that
\begin{equation*}
\Phi^{-1}(2ku)\le k\Phi^{-1}(u).
\end{equation*}
For $u\ge0$, taking $t = \Phi(ku)$, we have $ku \le \Phi^{-1}(\Phi(ku)) = \Phi^{-1}(t)$.
Hence
\begin{equation*}
\Phi(u) \le \Phi\left(\frac1k \Phi^{-1}(t)\right) \le\Phi\left(\Phi^{-1}\left(\frac{t}{2k}\right)\right) \le 
\frac{t}{2k} = \frac1{2k} \Phi(ku).
\end{equation*}
\end{proof}

\begin{lemma}\label{lem:210321-1}
Let $\Phi:[0,\infty]\to[0,\infty]$ be a Young function.
If $\Phi \in \Delta_2$, then there exists $q \in(0, \infty)$ and $C \in [1, \infty)$ such that
\begin{equation*}
\dfrac1t \Phi^{-1}\left(\dfrac1t\right)^{-q}
\le C
\dfrac1s \Phi^{-1}\left(\dfrac1s\right)^{-q}
\quad \text{for} \quad  t\le s.
\end{equation*}
Conversely, if there exists $q \in(0, \infty)$ and $C \in [1, \infty)$ such that
\begin{equation*}
\dfrac1t \Phi^{-1}\left(\dfrac1t\right)^{-q}
\le C\dfrac1s \Phi^{-1}\left(\dfrac1s\right)^{-q}
\quad \text{for} \quad  t\le s,
\end{equation*}
then $\Phi\in\Delta_2$.
\end{lemma}

\begin{proof}
If $\Phi\in\Delta_2$, putting $q = \log_2k$ and $C = k$,we calculate
\begin{align*}
\Phi^{-1}\left(\frac1t\right)
&=
\Phi^{-1}\left(k^{\log_k \frac{s}{t}  }\frac1s\right)
\ge
\Phi^{-1}\left(k^{\left[\log_k \frac{s}{t}\right]  }\frac1s\right)
\ge
2^{\left[\log_k \frac{s}{t}\right] } \Phi^{-1}\left( \frac1s\right)
\ge
2^{\log_k \frac{s}{t}-1 } \Phi^{-1}\left( \frac1s\right) \\
&=
\left(\left(\frac12\right)^{\frac1{\log_k2}}\frac{s}{t} \right)^{\log_k2}\Phi^{-1}\left( \frac1s\right)
=\left(\frac1{C}\frac{s}{t} \right)^{\frac1q}\Phi^{-1}\left( \frac1s\right),
\end{align*}
where $[\cdot]$ stands for the Gauss symbol.
This is the desired result.

Conversely, we suppose that there exists $q \in(0,\infty)$ and $C \in [1, \infty)$ such that 
\begin{equation*}
\frac1t \Phi^{-1}\left(\frac1t\right)^{-q}
\le C
\frac1s \Phi^{-1}\left(\frac1s\right)^{-q}
\quad \text{for} \quad  t\le s.
\end{equation*}
Then, by the change of variables $u=\dfrac1t$ and $v=\dfrac1s$, we compute
\begin{align*}
u\Phi^{-1}(u)^{-q} \le v \Phi^{-1}(v)^{-q},
\end{align*}
or equivalently
\begin{align*}
\Phi^{-1}(v) \le \left( \frac{v}{u} \right)^{\frac1q} \Phi^{-1}(u).
\end{align*}
Consequently, choosing $u=kv$ and $k = 2^qC >1$, we obtain
\begin{align*}
\Phi^{-1} \left(v\right)
\le
\left(\frac1k\right)^{\frac1q} \Phi^{-1}\left( kv\right) 
\le
\frac12 \Phi^{-1}\left( kv\right) 
\end{align*}
as desired by Lemma \ref{lem:201107-1} (1).
\end{proof}

\subsection{Orlicz-Lorentz spaces and weak-type Orlicz spaces}

\begin{definition}[Orlicz-Lorentz space]
For a parameter $0<q\le\infty$ and a Young function $\Phi:[0,\infty]\to[0,\infty]$, let
\begin{equation*}
L^{\Phi,q}({\mathbb R}^n)
\equiv
\left\{f\in L^0({\mathbb R}^n):\|f\|_{L^{\Phi,q}}<\infty\right\},
\end{equation*}
endowed with the quasi-norm
\begin{equation*}
\|f\|_{L^{\Phi,q}}
\equiv
\begin{cases}
\displaystyle
\left(
\int_0^\infty\left[\Phi^{-1}\left(\frac1t\right)^{-1}f^\ast(t)\right]^q\,\frac{{\rm d}t}t
\right)^{\frac1q}, & 0<q<\infty,\\
\displaystyle\sup_{t>0}\Phi^{-1}\left(\frac1t\right)^{-1}f^\ast(t), & q=\infty.
\end{cases}
\end{equation*}
\end{definition}

\begin{definition}[Orlicz space]
For a Young function $\Phi:[0,\infty]\to[0,\infty]$, let
\begin{equation*}
L^\Phi({\mathbb R}^n)
\equiv
\left\{
f\in L^0({\mathbb R}^n)
:
\int_{{\mathbb R}^n}\Phi(k|f(x)|)\,{\rm d}x<\infty\,\text{for some $k>0$}
\right\},
\end{equation*}
\begin{equation*}
\|f\|_{L^\Phi}
\equiv
\inf
\left\{
\lambda>0
:
\int_{{\mathbb R}^n}\Phi\left(\frac{|f(x)|}\lambda\right)\,{\rm d}x\le1
\right\}.
\end{equation*}
\end{definition}

\begin{definition}
Let $\Phi:[0,\infty]\to[0,\infty]$ be a Young function.
The weak-type Orlicz space ${\rm W}L^\Phi({\mathbb R}^n)$ is defined by
\begin{equation*}
{\rm W}L^\Phi({\mathbb R}^n)
\equiv
\left\{
f\in L^0({\mathbb R}^n)
\,:\,
\|f\|_{{\rm W}L^\Phi}\equiv\sup_{t>0}t\|\chi_{\{x\in{\mathbb R}^n\,:\,|f(x)|>t\}}\|_{L^\Phi}
\right\}.
\end{equation*}
\end{definition}

\begin{proposition}\label{prop:weak Orlicz space}
Let $\Phi:[0,\infty]\to[0,\infty]$ be a Young function.
Then we have ${\rm w}L^\Phi({\mathbb R}^n)={\rm W}L^\Phi({\mathbb R}^n)=L^{\Phi,\infty}({\mathbb R}^n)$.
\end{proposition}

\begin{proof}
See \cite[Proposition 1.4.5 (16)]{Grafakos14-1} for the equality ${\rm W}L^\Phi({\mathbb R}^n)=L^{\Phi,\infty}({\mathbb R}^n)$ in detail.

When we choose $f\in{\rm w}L^\Phi({\mathbb R}^n)$, we prove $\|f\|_{{\rm W}L^\Phi}\le\|f\|_{{\rm w}L^\Phi}$.
Fix a sufficiently small number $\varepsilon>0$.
Then
\begin{equation*}
\sup_{t>0}\Phi(t)m\left(\frac f{\|f\|_{{\rm w}L^\Phi}+\varepsilon},t\right)\le1.
\end{equation*}
Changing $t\mapsto\dfrac t{\|f\|_{{\rm w}L^\Phi}+\varepsilon}$, we have
\begin{equation*}
\sup_{t>0}\Phi\left(\frac t{\|f\|_{{\rm w}L^\Phi}+\varepsilon}\right)m(f,t)\le1.
\end{equation*}
By the equation (\ref{inverse equation}), for all $t>0$,
\begin{equation*}
\frac t{\|f\|_{{\rm w}L^\Phi}+\varepsilon}\le\Phi^{-1}\left(\frac1{m(f,t)}\right),
\end{equation*}
or equivalently,
\begin{equation*}
t\Phi^{-1}\left(\frac1{m(f,t)}\right)^{-1}\le\|f\|_{{\rm w}L^\Phi}+\varepsilon.
\end{equation*}
Thus we obtain
\begin{equation*}
\|f\|_{{\rm W}L^\Phi}\le\|f\|_{{\rm w}L^\Phi}+\varepsilon.
\end{equation*}

Let $f\in{\rm W}L^\Phi({\mathbb R}^n)$.
We verify $\|f\|_{{\rm W}L^\Phi}\ge\|f\|_{{\rm w}L^\Phi}$.
Remark that for all $t>0$,
\begin{equation*}
t\Phi^{-1}\left(\frac1{m(f, t)  }\right)^{-1}
\le
\|f\|_{{\rm W}L^\Phi}.
\end{equation*}
We calculate
\begin{align*}
\Phi\left(\frac t{\|f\|_{{\rm W}L^\Phi}}\right)
\le
\Phi\left(\Phi^{-1}\left(\frac1{ m(f, t) }\right)\right)
\le
\frac1{ m(f, t)  },
\end{align*}
and then
\begin{equation*}
\Phi\left(\frac t{\|f\|_{{\rm W}L^\Phi}}\right)m(f, t)
\le
1.
\end{equation*}
Hence
\begin{equation*}
\sup_{t>0}\Phi\left(\frac t{\|f\|_{{\rm W}L^\Phi}}\right)m(f, t) 
\le
1.
\end{equation*}
Therefore $\|f\|_{{\rm w}L^\Phi}\le\|f\|_{{\rm W}L^\Phi}$.
\end{proof}

\begin{proposition}\label{prop:201107}
Let $0<q<\infty$, and let $\Phi:[0,\infty]\to[0,\infty]$ be a Young function.
Then for any measurable set $E\subset{\mathbb R}^n$, the followings are hold:
\begin{itemize}
\item[(1)] $\|\chi_E\|_{L^{\Phi,q}}\gtrsim\Phi^{-1}\left(\dfrac1{|E|}\right)^{-1}$ and $\|\chi_E\|_{L^{\Phi,\infty}}=\Phi^{-1}\left(\dfrac1{|E|}\right)^{-1}$.
\item[(2)] $\Phi\in\Delta_2$ implies $\|\chi_E\|_{L^{\Phi,q}}\sim\Phi^{-1}\left(\dfrac1{|E|}\right)^{-1}$.
\end{itemize}
\end{proposition}

\begin{remark}
Let $\Phi$ be a Young function, and let $0<q<\infty$.
By the convexity of $\Phi$, $L^{\Phi,q}({\mathbb R}^n)\ne\{0\}$ if and only if
\begin{equation*}
\int_0^1\Phi^{-1}\left(\frac1t\right)^{-q}\,\frac{{\rm d}t}t
<\infty,
\end{equation*}
generally.
\begin{itemize}
\item[(1)] If $b(\Phi)<\infty$, then, we have $L^{\Phi,q}({\mathbb R}^n)=\{0\}$.
Indeed, it is suffices to show that for all measurable sets $E\subset{\mathbb R}^n$ with $|E|\ne0$,
\begin{equation*}
\|\chi_E\|_{L^{\Phi,1}}=\infty.
\end{equation*}
By definition, for $u\ge0$,
\begin{equation*}
\Phi^{-1}(u)
=
\inf\{0\le t\le b(\Phi)\,:\,\Phi(t)>u\}
\le b(\Phi)
\end{equation*}
Then,
\begin{align*}
\|\chi_E\|_{L^{\Phi,1}}
=
\int_0^{|E|}\Phi^{-1}\left(\frac1u\right)^{-1}\,\frac{{\rm d}u}u
\ge
\frac1{b(\Phi)}\int_0^{|E|}\frac{{\rm d}u}u
=
\infty.
\end{align*}
\item[(2)] Let 
\begin{equation*}
\Phi(t)=
\begin{cases}
0, & t\le1,\\
t-1, & t>1.
\end{cases}
\end{equation*}
Then, $a(\Phi)>0$ and
\begin{equation*}
\Phi^{-1}(u)
=
\inf\{t>1\,:\,t-1>u\}
=
u+1.
\end{equation*}
Additionally, for each measurable set $E\subset{\mathbb R}^n$,
\begin{align*}
\|\chi_E\|_{L^{\Phi,1}}
=
\int_0^{|E|}\left(\frac1u+1\right)^{-1}\,\frac{{\rm d}u}u
=
\int_0^{|E|}\frac1{1+u}\,{\rm d}u
=
\log(1+|E|)
\hspace{4pt}/\hspace{-9pt}\sim
\Phi^{-1}\left(\frac1{|E|}\right)^{-1}.
\end{align*}
\item[(3)] Taking $\Phi(t)=e^t-1$, $t\ge0$, we have $a(\Phi)=0$ and $b(\Phi)=\infty$.
Especially, $L^{\Phi,1}({\mathbb R}^n)=\{0\}$.
Indeed,
\begin{align*}
\Phi^{-1}(u)
=
\inf\{t\ge0\,:\,e^t-1>u\}
=
\log(1+u),
\end{align*}
and then, for any measurable set $E\subset{\mathbb R}^n$ with $|E|\le\dfrac12$,
\begin{align*}
\|\chi_E\|_{L^{\Phi,1}}
&=
\int_0^{|E|}\frac1{\log\left(1+\dfrac1u\right)}\,\frac{{\rm d}u}u
=
\int_0^{|E|}\frac1{\log(u+1)+\log\dfrac1u}\,\frac{{\rm d}u}u
\ge
\int_0^{|E|}\frac1{2\log\dfrac1u}\,\frac{{\rm d}u}u\\
&=
\infty.
\end{align*}
Meanwhile, $q>1$ implies $L^{\Phi,q}({\mathbb R}^n)\ne\{0\}$.
\end{itemize}
\end{remark}

\begin{proof}[Proof of Proposition \ref{prop:201107}]
(1) $\|\chi_E\|_{L^{\Phi,\infty}}=\Phi^{-1}\left(\dfrac1{|E|}\right)^{-1}$ is obtained by Proposition \ref{prop:weak Orlicz space} and the following calculation:
\begin{align*}
\|\chi_E\|_{L^{\Phi,\infty}}
=
\|\chi_E\|_{{\rm W}L^\Phi}
=
\|\chi_E\|_{L^\Phi}
=
\Phi^{-1}\left(\dfrac1{|E|}\right)^{-1}.
\end{align*}
In addition, since $\dfrac1t\Phi^{-1}\left(\dfrac1t\right)^{-1}$ is deceasing, we compute
\begin{align*}
\|\chi_E\|_{L^{\Phi,q}}
&=
\left(\int_0^{|E|}\Phi^{-1}\left(\frac1t\right)^{-q}\,\frac{{\rm d}t}t\right)^{\frac1q}
\ge
\frac1{|E|}\Phi^{-1}\left(\frac1{|E|}\right)^{-1}
\left(\int_0^{|E|}t^q\,\frac{{\rm d}t}t\right)^{\frac1q}\\
&=
\frac1{q^{\frac1q}}\Phi^{-1}\left(\frac1{|E|}\right)^{-1}
\end{align*}
as desired.

(2) We assume that $\Phi\in\Delta_2$.
By Lemma \ref{lem:210321-1}, there exists $q \in(0,\infty)$ and $C \in [1, \infty)$ such that 
\begin{equation*}
\dfrac1t \Phi^{-1}\left(\dfrac1t\right)^{-q}
\le C
\dfrac1s \Phi^{-1}\left(\dfrac1s\right)^{-q}
\quad \text{for} \quad  t\le s,
\end{equation*}
and then,
\begin{align*}
\|\chi_E\|_{L^{\Phi,q}}^q
&=
\int_0^{|E|}\Phi^{-1}\left(\frac1t\right)^{-q}\,\frac{{\rm d}t}t
\le
C\frac1{|E|}\Phi^{-1}\left(\frac1{|E|}\right)^{-q}\int_0^{|E|}\,{\rm d}t
=
C\Phi^{-1}\left(\frac1{|E|}\right)^{-q}.
\end{align*}
\end{proof}

\section{Proof of Theorem \ref{thm:201006-1}}\label{s3}

Given $T\in L^{\Phi,1}({\mathbb R}^n)^\ast$, we consider the measure $\mu(E)\equiv T\chi_E$.
Since $\mu$ satisfies
\begin{equation*}
|\mu(E)|
\le
\|T\|_{(L^{\Phi,1})^\ast}\|\chi_E\|_{L^{\Phi,1}}
\sim
\|T\|_{(L^{\Phi,1})^\ast}\Phi^{-1}\left(\frac1{|E|}\right)^{-1},
\end{equation*}
it follows that $\mu$ is absolutely continuous with respect to the Lebesgue measure $|\cdot|$.
By the Radon-Nykodym theorem, there exists a unique measurable function $g\in L^1({\mathbb R}^n)$ such that
\begin{equation*}
\mu(E)=\int_Eg(x)\,{\rm d}x
\end{equation*}
for all measurable sets $E\subset{\mathbb R}^n$ with $0<|E|<\infty$.
Therefore, we obtain
\begin{equation*}
T\chi_E=\int_{{\mathbb R}^n}\chi_E(x)g(x)\,{\rm d}x,
\quad \text{$0<|E|<\infty$, $E\subset{\mathbb R}^n$}.
\end{equation*}

In addition, we obtain
\begin{equation}\label{eq:201006-1}
\left|\int_{{\mathbb R}^n}f(x)g(x)\,{\rm d}x\right|
\le
\|T\|_{(L^{\Phi,1})^\ast}\|f\|_{L^{\Phi,1}}.
\end{equation}
In fact, given $f\in L^{\Phi,1}({\mathbb R}^n)$, we take a sequence of non-negative simple functions $\{f_j\}_{j\ge1}$ such that
\begin{equation*}
f_j\uparrow|f|, \quad
\mbox{a.e.}.
\end{equation*}
Then, by the Fatou lemma, we have
\begin{equation*}
\int_{{\mathbb R}^n}|f(x)g(x)|\,{\rm d}x
\le
\liminf_{j\to\infty}\int_{{\mathbb R}^n}|f_j(x)g(x)|\,{\rm d}x,
\quad g\in L^1({\mathbb R}^n).
\end{equation*}
We compute
\begin{align*}
\int_{{\mathbb R}^n}|f_j(x)g(x)|\,{\rm d}x
=
T[f_j\,{\rm sgn}\,g]
\le
\|T\|_{(L^{\Phi,1})^\ast}\|f_j\,{\rm sgn}\,g\|_{L^{\Phi,1}}
\le
\|T\|_{(L^{\Phi,1})^\ast}\|f\|_{L^{\Phi,1}}.
\end{align*}

Note that, for each $t>0$,
\begin{align*}
\left\|\frac{\bar{g}}{|g|}\chi_{\{x\in{\mathbb R}^n\,:\,|g(x)|>t\}}\right\|_{L^{\Phi,1}}
\sim
\Phi^{-1}\left(\frac1{|\{x\in{\mathbb R}^n\,:\,|g(x)|>t\}|}\right)^{-1}
\le
\Phi^{-1}\left(\frac t{\|g\|_{L^1}}\right)^{-1}
<\infty.
\end{align*}
Taking $f=\dfrac{\bar{g}}{|g|}\chi_{\{x\in{\mathbb R}^n\,:\,|g(x)|>t\}}\in L^{\Phi,1}({\mathbb R}^n)$ for each $t>0$ in (\ref{eq:201006-1}), we have
\begin{equation*}
t|\{x\in{\mathbb R}^n\,:\,|g(x)|>t\}|
\le
\|T\|_{(L^{\Phi,1})^\ast}\Phi^{-1}\left(\frac1{|\{x\in{\mathbb R}^n\,:\,|g(x)|>t\}|}\right)^{-1},
\end{equation*}
and, by (\ref{eq:201006-2}), then
\begin{equation*}
t\widetilde{\Phi}^{-1}\left(\frac1{|\{x\in{\mathbb R}^n\,:\,|g(x)|>t\}|}\right)^{-1}
\le
\|T\|_{(L^{\Phi,1})^\ast}.
\end{equation*}
Consequently, take the supremum over $t>0$ to obtain that
\begin{equation*}
\|g\|_{{\rm w}L^{\widetilde{\Phi}}}\le\|T\|_{(L^{\Phi,1})^\ast}.
\end{equation*}

Conversely, using Exercise 1.4.1 (b) in \cite{Grafakos14-1} and equation (\ref{eq:201006-2}), we see that if $f\in L^{\Phi,1}({\mathbb R}^n)$ and $h\in{\rm w}L^{\widetilde{\Phi}}({\mathbb R}^n)$, then
\begin{align*}
\int_{{\mathbb R}^n}|f(x)h(x)|\,{\rm d}x
&\le
\int_0^\infty f^\ast(t)h^\ast(t)\,{\rm d}t
\le
2
\int_0^\infty
\Phi^{-1}\left(\frac1t\right)^{-1}f^\ast(t)
\cdot
\widetilde{\Phi}^{-1}\left(\frac1t\right)^{-1}h^\ast(t)
\,\frac{{\rm d}t}t\\
&\le
2\|f\|_{L^{\Phi,1}}\|h\|_{{\rm w}L^{\widetilde{\Phi}}}.
\end{align*}
Thus every $h\in{\rm w}L^{\widetilde{\Phi}}({\mathbb R}^n)$ gives rise to a bounded linear functional $f\mapsto\int hf\,{\rm d}x$ on $L^{\Phi,1}({\mathbb R}^n)$ with norm at most $\|h\|_{{\rm w}L^{\widetilde{\Phi}}}$.

\section{Generalized Lorentz spaces}\label{s4}

\begin{definition}
Let $0<q\le\infty$, and let $\varphi:{\mathbb R}_+\to{\mathbb R}_+$ be a measurable function.
We define the generalized Lorentz space $\Lambda^{\varphi,q}({\mathbb R}^n)$ by the set of all measurable functions $f$ with the finite quasi-norm
\begin{equation*}
\|f\|_{\Lambda^{\varphi,q}}
\equiv
\begin{cases}
\displaystyle
\left(\int_0^\infty[\varphi(t)f^\ast(t)]^q\,\frac{{\rm d}t}t\right)^{\frac1q}, & q<\infty,\\
\displaystyle\underset{t>0}{\rm ess\,sup}\,\varphi(t)f^\ast(t), & q=\infty.
\end{cases}
\end{equation*}
\end{definition}

\begin{remark}
If $\varphi$ is a non-decreasing function satisfying the doubling condition; there exists $C\ge1$ such that for all $s,t>0$,
\begin{equation*}
\frac1C\le\frac{\varphi(t)}{\varphi(s)}\le C,
\quad \mbox{if} \quad
\frac12\le\frac ts\le2,
\end{equation*}
then, the generalized Lornetz space $L^{\varphi,q}({\mathbb R}^n)$ is vector space.
\end{remark}

\begin{remark}
\begin{itemize}
\item[(1)] If $\varphi$ is a non-increasing function and $q<\infty$, then, $\Lambda^{\varphi,q}({\mathbb R}^n)=\{0\}$.
\item[(2)] If $\varphi(t)=t^{\frac1p}$, then, $\Lambda^{\varphi,q}({\mathbb R}^n)$ is the classical Lorentz space $L^{p,q}({\mathbb R}^n)$.
\item[(3)] If $\varphi(t)=\Phi^{-1}\left(\dfrac1t\right)^{-1}$, then, $\Lambda^{\varphi,q}({\mathbb R}^n)$ is the Orlicz-Lorentz space $L^{\Phi,q}({\mathbb R}^n)$.
\end{itemize}
\end{remark}

\subsection{Boundedness on Hardy-Littlewood maximal operator on generalized Lorentz spaces}

\begin{theorem}[{\cite[Corollary 1.9]{ArMu90}}]\label{cor:ArMu90-1.9}
Let $0<q<\infty$, and let $\varphi:{\mathbb R}_+\to{\mathbb R}_+$ be a measurable function.
Then the Hardy-Littlewood maximal operator $M$ is bounded on $\Lambda^{\varphi,q}({\mathbb R}^n)$ if and if for every $r>0$,
\begin{equation}\label{eq:ArMu90-1.9}
\int_r^\infty\left(\frac{\varphi(t)}t\right)^q\,\frac{{\rm d}t}t
\lesssim
\frac1{r^q}\int_0^r\varphi(t)^q\,\frac{{\rm d}t}t.
\end{equation}
\end{theorem}

\begin{lemma}\label{lem:201027-1}
Let $0<q <\infty$, and $\Phi$ be a Young function.
If $\Phi\in\nabla_2$, then $\varphi(t)\equiv\Phi^{-1}\left(\dfrac1t\right)^{-1}$ satisfies the equation (\ref{eq:ArMu90-1.9}).
\end{lemma}

\begin{proof}
Let $q<\infty$.

Fix $r>0$.
Using Lemma \ref{lem:201107-1} (2), we calculate
\begin{align*}
\int_r^\infty\left[\frac1t\Phi^{-1}\left(\frac1t\right)^{-1}\right]^q\,\frac{{\rm d}t}t
&=
\int_1^\infty\left[\frac1{rt}\Phi^{-1}\left(\frac1{rt}\right)^{-1}\right]^q\,\frac{{\rm d}t}t
\le
\sum_{j=1}^\infty
\int_{(2k)^{j-1}}^{(2k)^j}
\left[\frac1{rt}\Phi^{-1}\left(\frac1{r(2k)^j}\right)^{-1}\right]^q
\,\frac{{\rm d}t}t\\
&\le
\frac1{r^q}\Phi^{-1}\left(\frac1r\right)^{-q}
\sum_{j=1}^\infty k^{q(j-1)}\cdot\frac1q\left[\frac1{(2k)^{qj}}-\frac1{(2k)^{qj}}\right]
\sim
\frac1{r^q}\Phi^{-1}\left(\frac1r\right)^{-q},
\end{align*}
and
\begin{align*}
\int_0^r\Phi^{-1}\left(\frac1t\right)^{-q}\,\frac{{\rm d}t}t
&=
\int_0^1\Phi^{-1}\left(\frac1{rt}\right)^{-q}\,\frac{{\rm d}t}t
\ge
\sum_{j=1}^\infty
\int_{\frac1{(2k)^j}}^{\frac1{(2k)^{j-1}}}\Phi^{-1}\left(\frac{(2k)^{j-1}}r\right)^{-q}\,\frac{{\rm d}t}t\\
&\ge
\sum_{j=1}^\infty\frac1{k^{q(j-1)}}\Phi^{-1}\left(\frac1r\right)^{-q}
\left(\log\frac1{k^{j-1}}-\log\frac1{k^j}\right)
\sim
\Phi^{-1}\left(\frac1r\right)^{-q}.
\end{align*}
Then we obtain
\begin{equation*}
\left.
\int_r^\infty\left[\frac1t\Phi^{-1}\left(\frac1t\right)^{-1}\right]^q\,\frac{{\rm d}t}t
\right/
\frac1{r^q}\int_0^r\Phi^{-1}\left(\frac1t\right)^{-q}\,\frac{{\rm d}t}t
\lesssim
1.
\end{equation*}
\end{proof}

\begin{theorem}\label{thm:201206-1}
Let $\varphi:{\mathbb R}_+\to{\mathbb R}_+$ be a measurable function.
Then the Hardy-Littlewood maximal operator $M$ is bounded on $\Lambda^{\varphi,\infty}({\mathbb R}^n)$ if and if for every $r>0$,
\begin{equation}\label{eq:201206-2}
\underset{t>0}{\rm ess\,sup}\,\frac{\varphi(t)}t\int_0^t\frac{{\rm d}s}{\underset{0<\tau<s}{\rm ess\,sup}\,\varphi(\tau)}
<\infty.
\end{equation}
\end{theorem}

To prove this theorem, we may use the following lemmas.

\begin{lemma}[{\cite[Theorem 4.7]{KGS14}}]\label{lem:201206-1}
Let $v,w:{\mathbb R}_+\to{\mathbb R}_+$ be measurable functions.
Then the inequality
\begin{equation*}
\underset{t>0}{\rm ess\,sup}\,w(t)\int_0^tg(s)\,{\rm d}s
\lesssim
\underset{t>0}{\rm ess\,sup}\,v(t)g(t)
\end{equation*}
holds for all non-negative and non-increasing $g$ on $(0,\infty)$ if and only if
\begin{equation*}
\underset{t>0}{\rm ess\,sup}\,w(t)\int_0^t\frac{{\rm d}s}{\underset{0<\tau<s}{\rm ess\,sup}\,v(\tau)}
<\infty.
\end{equation*}
\end{lemma}

\begin{lemma}\label{lem:201206-2}
Let $\varphi:{\mathbb R}_+\to{\mathbb R}_+$ be a measurable function.
Then the Hardy-Littlewood maximal operator $M$ is bounded on $\Lambda^{\varphi,\infty}({\mathbb R}^n)$ if and if for all non-negative and non-increasing $g$ on ${\mathbb R}_+$,
\begin{equation}\label{eq:201206-1}
\underset{t>0}{\rm ess\,sup}\,\frac{\varphi(t)}t\int_0^tg(s)\,{\rm d}s
\lesssim
\underset{t>0}{\rm ess\,sup}\,\varphi(t)g(t).
\end{equation}
\end{lemma}

To prove this lemma, we use the following lemma:

\begin{lemma}\label{lem:210314-1}
Let $g:(0,\infty)\to(0,\infty)$ be a right-continuous non-increasing function.
Then, taking $f(x)\equiv g(\nu_n|x|^n)$ for $x\in{\mathbb R}^n$, where $\nu_n$ be a volume of the $n$-dimensional unit ball, we have $f^\ast(t)=g(t)$ for $t>0$.
\end{lemma}

\begin{proof}
For any $\lambda>0$,
\begin{align*}
m(f,\lambda)
&=
|\{x\in{\mathbb R}^n\,:\,g(\nu_n|x|^n)>\lambda\}|
=
\nu_n\sup\{s>0\,:\,g(\nu_ns^n)>\lambda\}^n\\
&=
\sup\{s>0\,:\,g(s)>\lambda\}.
\end{align*}
Then, we note that
\begin{equation*}
f^\ast(t)
=
\inf\{\lambda>0\,:\,\sup\{s>0\,:\,g(s)>\lambda\}\le t\}.
\end{equation*}

Fix $t>0$.
By the non-increasingly of $g$,
\begin{equation*}
\sup\{s>0\,:\,g(s)>g(t)\}\le t,
\end{equation*}
and then, $f^\ast(t)\le g(t)$.
Meanwhile, by the non-increasingly and right-continuity of $g$, for all sufficiently small number $\varepsilon>0$, there exists $t_0>0$ such that $g(t_0)\ge g(t)-\varepsilon$.
It follows that
\begin{align*}
\sup\{s>0\,:\,g(s)>g(t)-\varepsilon\}\ge t_0>t,
\end{align*}
and therefore, $f^\ast(t)\ge g(t)-\varepsilon$.
This is the desired result.
\end{proof}

\begin{proof}[Proof of Lemma \ref{lem:201206-2}]
Similar to show this lemma by the proof of Lemma 4.1 in \cite{ArMu90}, and we use the fact that
\begin{equation*}
(Mf)^\ast(t)
\sim
\frac1t\int_0^tf^\ast(s)\,{\rm d}s,
\quad t>0,
\end{equation*}
for measurable function $f$ (see for example \cite[p. 306]{Zygmund59}).

We assume that the Hardy-Littlewood maximal operator $M$ is bounded on $\Lambda^{\varphi,\infty}({\mathbb R}^n)$.
Fix a non-negative non-increasing function $g$ on ${\mathbb R}_+$ and define $f\in L^0({\mathbb R}^n)$ by $f(x)=g(\nu_n|x|^n)$.
Then by Lemma \ref{lem:210314-1},
\begin{equation*}
f^\ast(t)=g(t),
\end{equation*}
and then we can verify the equation (\ref{eq:201206-1}), immediately.

Conversely, taking non-increasing $g$ in (\ref{eq:201206-1}) by $f^\ast$, we have
\begin{equation*}
\underset{t>0}{\rm ess\,sup}\,\varphi(t)f^\ast(t)
\gtrsim
\underset{t>0}{\rm ess\,sup}\,\frac{\varphi(t)}t\int_0^tf^\ast(s)\,{\rm d}s
\sim
\underset{t>0}{\rm ess\,sup}\,\varphi(t)(Mf)^\ast(t)
\end{equation*}
for all $f\in\Lambda^{\varphi,\infty}({\mathbb R}^n)$.
This is the desired result.
\end{proof}

\begin{lemma}\label{lem:210311-1}
Let $\Phi:[0,\infty]\to[0,\infty]$ be a Young function in $\nabla_2$.
Then $\varphi(t)=\Phi\left(\dfrac1t\right)^{-1}$, $t>0$, satisfies the equation (\ref{eq:201206-2}).
\end{lemma}

\begin{proof}
Fix $t>0$.
Using Lemma \ref{lem:201107-1} (2), we calculate
\begin{align*}
\frac1t\Phi^{-1}\left(\frac1t\right)^{-1}
\int_0^t
\frac{{\rm d}s}{\underset{0<\tau<s}{\rm ess\,sup}\,\Phi^{-1}\left(\dfrac1\tau\right)^{-1}}
&=
\frac1t\Phi^{-1}\left(\frac1t\right)^{-1}
\int_0^t\Phi^{-1}\left(\frac1s\right)\,{\rm d}s\\
&=
\Phi^{-1}\left(\frac1t\right)^{-1}
\int_0^1\Phi^{-1}\left(\frac1{ts}\right)\,{\rm d}s\\
&\le
\Phi^{-1}\left(\frac1t\right)^{-1}
\sum_{j=1}^\infty\int_{\frac1{(2k)^j}}^{\frac1{(2k)^{j-1}}}
\Phi^{-1}\left(\frac{(2k)^{j-1}}t\right)\,{\rm d}s\\
&\le
\Phi^{-1}\left(\frac1t\right)^{-1}
\sum_{j=1}^\infty
\left(\frac1{(2k)^{j-1}}-\frac1{(2k)^j}\right)k^{j-1}\Phi^{-1}\left(\frac1t\right)\\
&\lesssim
1.
\end{align*}
\end{proof}

\begin{theorem}\label{prop:201027-1}
Let $0<q\le \infty$, and let $\Phi$ be a Young function in $\nabla_2$.
Then the Hardy-Littlewood maximal operator $M$ is bounded on $L^{\Phi,q}({\mathbb R}^n)$.
In particular, the Hardy-Littlewood maximal operator $M$ is bounded on ${\rm w}L^\Phi({\mathbb R}^n)$.
\end{theorem}

\section{Proof of Theorem \ref{thm:201007-1}}\label{s5}

\begin{definition}
Let $\Phi:[0,\infty]\to[0,\infty]$ be a Young function. We define
\begin{align*}
	p_+=p_+(\Phi)&:=\inf \left\{1\leq p\leq\infty\,\mid\,\Phi(\lambda r)\leq \lambda^p\Phi(r)\ \ \text{for}\ r\geq0,\ \lambda>1\right\}\\
	p_-=p_-(\Phi)&:=\sup \left\{1\leq p\leq\infty\,\mid\,\Phi(\lambda r)\leq \lambda^p\Phi(r)\ \ \text{for}\ r\geq0,\ 0<\lambda < 1\right\}.
\end {align*}
\end{definition}

\begin{lemma}\label{lem:210131-1}
Let $\Phi$ be a Young function.
Then we have
\begin{itemize}
\item[{\rm (1)}]
If $p_+<\infty$, then we have
$\tilde{p}_-:=p_-(\widetilde{\Phi})\geq  p_+'$,
\item[{\rm (2)}] 
If $p_->1$, then we have
$\tilde{p}_+:=p_+(\widetilde{\Phi})\leq  p_-'$.
\end{itemize}
\end{lemma}

\begin{proof}
$(1)$ We assume that $\Phi(\lambda r)\leq \lambda^{p_+}\Phi(r)$ for any $\lambda>1$ and $r>0$. From the definition of $\widetilde{\Phi}$, we have
\begin{align*}
\widetilde{\Phi}(r)
&=
\sup_{s>0}\{sr-\Phi(s)\}
\leq
\sup_{s>0}\{sr-\lambda^{-p_+}\Phi(\lambda s)\}
=
\lambda^{-p_+}\sup_{s>0}\{\lambda^{p_+}sr-\Phi(\lambda s)\}\\
&=
\lambda^{-p_+}\widetilde{\Phi}(\lambda^{p_+-1}r).
\end{align*}
When
\begin{equation*}
\Phi(\lambda r)\leq \lambda^{p_+}\Phi(r), \quad r\ge0,
\end{equation*}
by the change of variables $s=\lambda^{p_+-1}r$ and $\mu=\lambda^{1-p_+}$, we have
\[
\widetilde{\Phi}(\mu s)\leq \mu^{p'_+}\widetilde{\Phi}(s).
\]
Keeping in mind that $\lambda>1$ if and only if $0<\mu<1$, we have the desired result. \\
$(2)$ We assume that $\Phi(\lambda r)\leq \lambda^{p_-}\Phi(r)$ for any $0<\lambda< 1$ and $r>0$. So far, if
\begin{equation*}
\Phi(\lambda r)\leq \lambda^{p_-}\Phi(r),
\end{equation*}
then we have
\[
\widetilde{\Phi}(\mu s)\leq \mu^{p'_-}\widetilde{\Phi}(s),
\]
where $\lambda^{p_--1}r=s$ and $\lambda^{1-p_-}=\mu$. Since $0<\lambda<1$ if and only if $\mu>1$, we have the conclusion.
\end{proof} 

From Lemma \ref{lem:210131-1}, we have the following corollary.

\begin{lemma}\label{lem:210118-2}
Let Let $\Phi:[0,\infty]\to[0,\infty]$ be a Young function. 
\begin{itemize}
\item[{\rm (1)}]
$p_+<\infty$ if and only if $\Phi\in \Delta_2$.
\item[{\rm (2)}]
$p_->1$ if and only if $\Phi\in\nabla_2$.
\end{itemize}
\end{lemma} 

\begin{proof}
We assume $p_+<\infty$. Letting $\lambda=2$ in the definition of $p_+$, we have
\[
\Phi(2r)\leq 2^{p_+}\Phi(r).
\]
From $p_+<\infty$, we get $1<2^{p_+}<\infty$. Meanwhile, we assume $\Phi\in \Delta_2$. For $t>0$, we compute\\
 $t\Phi'(t)\leq \Phi(2t)\leq k\Phi(t)$, where $k>1$ is a constant appeared in the definition of the condition $\Delta_2$. Which implies 
 \[
 \frac{\Phi'(t)}{\Phi(t)}\leq k\frac{1}{t}.
 \]
 Let $\lambda>1$. Integrating both sides of the above inequality from $r$ to $\lambda r$, we have
\[
 \int_r^{\lambda r}\frac{\Phi'(t)}{\Phi(t)}dt\leq k\int_r^{\lambda r}\frac{dt}{t}
 \]
 \[
 \log\frac{\Phi(\lambda r)}{\Phi(r)}\leq k\log\frac{\lambda r}{r}
 \]
 \[
\Phi(\lambda r)\leq \lambda^k\Phi(r).
\]
 Thus, we get $p_+<k<\infty$.
Now, we turn to prove (2). From Lemma \ref{lem:210131-1} and (1), we get that the conditions $p_->1$ and $\Phi\in\nabla_2$ are equivalent.
\end{proof}

\begin{lemma}\label{lem:210403}
Let $\Phi$ be a young function and $\dfrac{1}{p_-(\Phi)}\leq\theta<\infty$. We define
\[
\Phi_{\theta}(r)=\int_0^{r^{\theta}}\frac{\Phi(t)}{t}dt.
\]
Then, $\Phi_{\theta}(r)$ is a Young function and we have
\begin{equation}\label{eq:210403-1}
\theta p_-(\Phi)\leq p_-(\Phi_{\theta})\leq p_+(\Phi_{\theta})\leq \theta p_+(\Phi).
\end{equation}
Moreover, we have $\|\cdot\|_{L^{\Phi}}\sim \|\cdot\|_{L^{\Phi_1}}$, and $\|\cdot\|_{{\rm w}L^{\Phi}}\sim \|\cdot\|_{{\rm w}L^{\Phi_1}}$.
\end{lemma}

\begin{proof}
By the change of variables, we have
\[
\Phi_{\theta}(r)=\int_0^{r^{\theta}}\frac{\Phi(t)}{t}dt
=\int_0^r\theta\frac{\Phi(t^{\theta})}{t}dt,
\]
which means that $\theta\Phi(t^{\theta})t^{-1}$ is the primitive function of $\Phi_{\theta}$. Since $\dfrac{1}{p_-(\Phi)}\leq \theta<\infty$ and $\Phi(\lambda t)\leq \lambda^{p_-}\Phi(t)$ for any $t>0$ and any $0<\lambda<1$, we get
\[
\frac{\Phi((\lambda t)^{\theta})}{\lambda t}\leq \frac{\lambda^{\theta p_-}\Phi(t^{\theta})}{\lambda t}\leq \lambda^{\theta p_--1}\frac{\Phi(t^{\theta})}{t}\leq \frac{\Phi(t^{\theta})}{t},
\]
for any $t>0$ and any $0<\lambda<1$. Thus, we obtain the fact that the primitive function $\theta\Phi(t^{\theta})t^{-1}$ is non-decreasing, which implies $\Phi_{\theta}$ is a Young function.
We assume $\Phi(\lambda t)\leq \lambda^p\Phi(t)$ for any $t>0$
and any $0<\lambda<1$ (resp. $\lambda>1$). Then, we get
\begin{align*}
\Phi_\theta(\lambda r)
=
\int_0^{(\lambda r)^{\theta}}\frac{\Phi(t)}{t}\,{\rm d}t
=
\int_0^{r^{\theta}}\frac{\Phi(\lambda^\theta t)}{t}\,{\rm d}t
\le
\lambda^{\theta p}\int_0^{r^{\theta}}\frac{\Phi(t)}{t}\,{\rm d}t
=
\lambda^{\theta p}\Phi_{\theta}(r),
\end{align*}
for any $r>0$ and any $0<\lambda<1$ (resp. $\lambda>1$). Which concludes $(\ref{eq:210403-1})$. From $\dfrac{\Phi(t)}{t}\leq \Phi'(t)\leq \dfrac{\Phi(2t)}{t}$ for $t>0$, it is easy to show that
\[
\Phi_1(r)\leq \Phi(r)\leq \Phi_1(2r),
\]
for $r>0$. Thus, we have $\|\cdot\|_{L^{\Phi}}\sim \|\cdot\|_{L^{\Phi_1}}$ and using Proposition 2.4, we also have $\|\cdot\|_{{\rm W}L^{\Phi}}\sim\|\cdot\|_{{\rm w}L^{\Phi}}\sim\|\cdot\|_{{\rm w}L^{\Phi_1}}$.
\end{proof}
Remark that $\Phi_{\theta}(r)=\Phi_1(r^{\theta})$.

Here we start the Proof of Theorem \ref{thm:201007-1}.

(1) If $q=\infty$, by the pointwise estimate
\begin{equation*}
Mf_k(x)
\le
M\left[\sup_{j\in{\mathbb N}}|f_j|\right](x),
\quad k\in{\mathbb N},
\end{equation*}
then this is an easy consequence of the boundedness of $M$
on ${\rm w}L^\Phi({\mathbb R}^n)$ (Proposition \ref{prop:201027-1}).

(2) Let $q<\infty$.
From $\Phi\in\Delta_2\cap\nabla_2$, we have $1<p_-(\Phi)\le p_+(\Phi)<\infty$. Let $\Psi=\Phi_{\frac1\eta}$ for $\eta\in(1,p_-(\Phi))$. Then, $\Psi$ is a Young function from Lemma \ref{lem:210403}. We fix a measurable non-negative function $\varphi$ in $L^{\widetilde{\Psi},1}({\mathbb R}^n)$ such that $\|\varphi\|_{L^{\widetilde{\Psi},1}}=1$. Thus by duality and 
\[
\left|\left|\left(\sum_{j=1}^\infty Mf_j(x)^q\right)^{\frac\eta q}\right|\right|_{{\rm w}L^{\Psi}}
=
\left|\left|\left(\sum_{j=1}^\infty Mf_j(x)^q\right)^{\frac 1q}\right|\right|^{\eta}_{{\rm w}L^{\Phi_1}}
\sim
\left|\left|\left(\sum_{j=1}^\infty Mf_j(x)^q\right)^{\frac 1q}\right|\right|^{\eta}_{{\rm w}L^{\Phi}},
\] 
it suffices to show that
\begin{equation*}
\int_{{\mathbb R}^n}\left(\sum_{j=1}^\infty Mf_j(x)^q\right)^{\frac \eta q}\varphi(x)\,{\rm d}x
\lesssim
\left\|\left(\sum_{j=1}^\infty|f_j|^q\right)^{\frac1q}\right\|_{{\rm w}L^\Phi}^{\eta}.
\end{equation*}
Now, choosing $\theta$ so that
\[
0<\theta<1, \quad
\left(\frac{p_+(\Phi)}\eta\right)'\theta>1.
\]
We note that $\widetilde{\Psi}_{\theta}\in\nabla_2$. In fact, from Lemma \ref{lem:210131-1} and Lemma \ref{lem:210403}, we obtain
\[
p_-(\widetilde{\Psi}_{\theta})\geq \theta\cdot p_-(\widetilde{\Psi})
\geq \theta\cdot p'_+(\Phi_{\frac1\eta}) \geq
\theta\cdot\left(\frac{p_+(\Phi)}{\eta}\right)'>1.
\]
Consequently, we obtain
\begin{align*}
\int_{{\mathbb R}^n}\left(\sum_{j=1}^\infty Mf_j(x)^q\right)^{\frac \eta q}\varphi(x)\,{\rm d}x
&\le
\int_{{\mathbb R}^n}
\left(\sum_{j=1}^\infty Mf_j(x)^q\right)^{\frac \eta q}
M\left[\varphi^{\frac1\theta}\right](x)^\theta
\,{\rm d}x\\
&\lesssim
\int_{{\mathbb R}^n}
\left(\sum_{j=1}^\infty|f_j(x)|^q\right)^{\frac \eta q}
M\left[\varphi^{\frac1\theta}\right](x)^\theta
\,{\rm d}x\\
&\lesssim
\left\|\left(\sum_{j=1}^\infty|f_j|^q\right)^{\frac \eta q}\right\|_{{\rm w}L^\Psi}
\left\|\left(M\left[\varphi^{\frac1\theta}\right]\right)^\theta\right\|_{L^{\widetilde{\Psi},1}}\\
&\lesssim
\left\|\left(\sum_{j=1}^\infty|f_j|^q\right)^{\frac1q}\right\|_{{\rm w}L^{\Phi_1}}^\eta
\sim \left\|\left(\sum_{j=1}^\infty|f_j|^q\right)^{\frac1q}\right\|_{{\rm w}L^\Phi}^\eta,
\end{align*}
where in the second inequality we used \cite[Theorem 3.1]{AnJo80/81}, since $(M[\varphi^{1/\theta}])^\theta\in A_\eta$ (see \cite[Theorem 7.7]{Duoandikoetxea01}).

\section*{Acknowledgement}
The authors would like to thank Professor Yoshihiro Sawano for their useful comments and pointing some typos.

\end{document}